\documentclass[11pt,reqno]{amsart}
\usepackage{amsmath,amsthm,amsfonts,amssymb,mathrsfs,stmaryrd,color}
\usepackage[all]{xy}
\usepackage{url}
\usepackage{geometry}
\usepackage{amsmath}
\usepackage{amsthm}
\usepackage{amsfonts}
\usepackage{amssymb}
\usepackage{tikz-cd}
\usepackage[numbers]{natbib}
\usepackage{mathrsfs}
\usepackage{enumerate}
\usepackage[utf8]{inputenc}
\usepackage[T1]{fontenc}
\usepackage{hyperref}
\usepackage[capitalize]{cleveref}
\usepackage{enumitem}
\setlist[enumerate]{itemsep=2pt,parsep=2pt,before={\parskip=2pt}}
\let\amsamp=&
\makeatletter
\newcommand{\colim@}[2]{%
  \vtop{\m@th\ialign{##\cr
    \hfil$#1\operator@font colim$\hfil\cr
    \noalign{\nointerlineskip\kern1.5\ex@}#2\cr
    \noalign{\nointerlineskip\kern-\ex@}\cr}}%
}
\newtheorem{theorem}{Theorem}[section]
\newtheorem*{theorem*}{Theorem}
\newtheorem*{definition*}{Definition}
\newtheorem{proposition}[theorem]{Proposition}
\newtheorem{lemma}[theorem]{Lemma}
\newtheorem{corollary}[theorem]{Corollary}

\theoremstyle{definition}
\newtheorem{definition}[theorem]{Definition}

\newtheorem{example}[theorem]{Example}

\newtheorem{claim}[theorem]{Claim}

\renewcommand{\lim}{\varprojlim}
\newcommand{\coker}[1]{\text{coker}(#1)}
\renewcommand{\to}{\rightarrow}

\renewcommand{\over}[2]{\stackrel{#1}{#2}}

\newcommand{\otimesl}{\otimes^{\textbf{L}}}

\newcommand{\bb}[1]{\mathbb{#1}}
\newcommand{\colim}[1]{\underset{#1}{\text{colim}} \;}

\newcommand{\cl}[1]{\text{cl}(#1)}
\begin{document}
\title{Indivisible sequences and descendability}
\author{Ivan Zelich}
\begin{abstract}
We introduce the notion of \textit{indivisible sequences} and show that to any indivisible sequence $\{S, \Psi: S \to R\}$ we can associate faithfully flat ring maps $R \to R'$ that are not descendable. As a corollary, we obtain the first example of a faithfully flat ring map between $\aleph_{n-1}$-countable rings that has descendability exponent $n$, and indeed a faithfully flat ring map between $\aleph_{\omega}$-countable rings that is not descendable.
\end{abstract}
\maketitle
\section{Introduction}
This paper analyses \textit{indivisible sequences}, a structure that proved useful in constructing faithfully flat ring maps that were not descendable, see \cite{zelich}. In particular, our main theorem can be summarised as follows.
\begin{theorem}[See Theorem~\ref{theorem:mainindiv}]\label{intro1}
Let $R$ be a ring that admits an an indivisible sequence $\{\{S_i, \Psi_i: S_i \to R\}\}_{i \in \bb{N}}$. Then if $|S_i| \ge \aleph_{i-1}$ for all $i \in \bb{N}$, there exists a faithfully flat $R$-algebra map $A \to A'$ that is not descendable with $|A|=|A'|=|R|$.
\end{theorem}
As a corollary, we obtain:
\begin{theorem}[See Corollary~\ref{corollary:mainindiv}]\label{intro2}
For any $n \in \bb{N} \cup \{\infty\}$, there exists a faithfully flat ring map $A \to A'$ between $\aleph_{n-1}$-Noetherian rings, with $A$ of Krull-dimension $n$, which has descendability exponent $n$. \\
\indent Furthermore, there exists a faithfully flat ring map $A \to A'$ between $\text{min}(\beth_{n-1},2^{\aleph_{n-1}})$-countable $p$-boolean rings which has descendability exponent $n$.
\end{theorem}
Our first result is optimal in light of the fact that flat modules over $\aleph_{k}$-Noetherian rings have projective dimension $\le k$ \cite[Th\'{e}or\`{e}m 7.10]{GrusonFlat}, \cite[Lemma D.3.3.7]{LurieSGA}.\\
\indent Our second result would be optimal if we could arrange for the map $A \to A'$ to be between $p$-boolean rings of cardinality $ \le \aleph_{n-1}$. If we assume the generalised continuum hypothesis, i.e. that $\beth_{n-1}=\aleph_{n-1}$, then we obtain an optimal result. Relatedly, a recent result of Aoki (\cite{aoki}) has shown that, in ZFC, there exist faithfully flat ring maps $A \to A'$ between $p$-boolean rings of cardinality $< \aleph_{2n}$ that have descendability exponent $n$. It seems therefore plausible that the optimal bound should be achieveable in ZFC.
\subsection*{Acknwledgements}
I would like to thank my advisor Aise Johan De Jong for our many discussions, feedback on drafts of this paper, and for his constant support. I would also like to thank Juan Esteban Rodriguez Camargo for helpful discussions, especially with Proposition~\ref{proposition:symformula}.
\subsection*{Notation}
We fix some notation that will be used throughout this article.
\begin{enumerate}\label{notation}
\item[(i)] If $S$ is a set and $R$ a ring, define $\text{Hom}^{\text{fin}}_{\text{Set}}(S, R) \subset \text{Hom}_{\text{Set}}(S,R)$ to be the set-theoretic maps with finite support i.e. $f \in \text{Hom}^{\text{fin}}_{\text{Set}}(S, R)$ if there exists a subset $K \subset S$ with $|K| < \infty$ such that $f(S\setminus K) = 0$. For each $t \in S$, set $\delta_{s=t} \in \text{Hom}^{\text{fin}}_{\text{Set}}(S, R)$ to be the function that is $1$ if $s=t$ and $0$ else.
\item[(ii)]Further, if $S_i$ are sets indexed by the set $\{1,2,...,n\}$, then we will identify the vector $1_{s_1,s_2,...,s_n}$ in the nested direct sum $\oplus_{S_1} \oplus_{S_2} \oplus ... \oplus_{S_n} R$ as follows: Every vector in $v \in \oplus_{S_1} \oplus_{S_2} \oplus ... \oplus_{S_n} R$ can be identified as a function $f_{v} \in \text{Hom}^{\text{fin}}_{\text{Set}}(\prod_i S_i, R)$, and under this identification, $1_{s_1,...,s_n}$ corresponds to $\delta_{t=s}$ where $s=(s_1,s_2,...,s_n) \in \prod_i S_i$.
\item[(iii)]For any set $S$, we let $p^n_i: S^{\times n} \to S^{\times n-1}$ be defined as follows: For any $s = (s_1,s_2,..,s_n) \in S^{\times n}$,
\[p^n_i(s) = (s_1,s_2,...,s_{i-1},s_{i+1},...,s_n).\]
For a point $s = (s_1,...,s_{n-1}) \in S^{\times n-1}$, we consider the degeneracy map $t_i^{s}: S \to S^{\times n}$ which takes a point $s' \in S$ to the point:
\[(s_1, ..., s_{i-1}, s', s_{i},...,s_{n-1}) \in S^{\times n}.\]
\item[(iv)] For a ring $R$ be a ring and as set $S$, let $R[S]=\text{Sym}_R(\oplus_{s \in S} R) = R[X_s]_{s \in S}$ and $R\{S\} = R[X_s]_{s \in S}/I$ where $I$ is the ideal generated by $X_s^2 - X_s \in R[S]$. \\
\indent Let $B(S) = \{K \subset S| \; |K| < \infty\}$, which can be endowed with the structure of a group under $\cup$, with unit $\emptyset$, which moreover satisfies $x \cup x = x$ for any $x \in B(S)$. Then $R\{S\} \simeq R[B(S)]$, the group ring on $B(S)$, and so we have an $R$-linear decomposition $R\{S\} = \oplus_n R\{S\}_{=n}$, where $R\{S\}_{=n} = \oplus_{\underset{|K|=n}{K\subset S}} R$. Let $\iota_S: S\to B(S)$ be the set-theoretic inclusion viewing elements of $S$ as one-element subsets of $S$, and let $R(\iota_S): \oplus_S R \to R\{S\}$ be the corresponding identification of $\oplus_S R$ as the degree $1$ elements of $R\{S\}$. \\
\indent We remark that $B(S \amalg T) = B(S) \times B(T)$ and therefore $R\{S \amalg T\} \simeq R\{S\} \otimes_R R\{T\}$.
\item[(v)] For an (injective) $A$-linear map $f: M \to N$ of $A$-modules, we may form an exact sequence 
\[0 \to M \to N \to \coker{f} \to 0.\]
This corresponds to a class in $\text{Ext}^1_A(\coker{f}, M)$, which we denote by $\cl{f}$.\\
\indent For a ring map $f: A \to B$ that is faithfully flat, we say that $f$ has \textit{descendability exponent $n$} when $\text{cl}(f)^{\otimes_A n} \neq 0$ but $\text{cl}(f)^{\otimes_A n+1} = 0$.
\end{enumerate}
Let $A$ is a commutative ring. The derived category $D(A)$ of $A$-modules, regarded as a stable $\infty$-category, admits a symmetric monoidal structure given by $\otimesl_A$. Therefore, given $A$-modules $M,N,M',N'$, there is an $A$-bilinear pairing:
\[\text{RHom}_A(M,N) \times \text{RHom}_A(M',N') \to \text{RHom}_A(M \otimesl_A M', N \otimesl_A N').\]
Passing to the homotopy category, we obtain an $A$-bilinear pairing of $A$-modules:
\[\text{Ext}^0_A(M,N) \times \text{Ext}^0_A(M',N') \to \text{Ext}^0_A(M \otimesl_A M', N \otimesl_A N'),\]
and for a pair of $f \in \text{Ext}^0_A(M,N), g \in \text{Ext}^0_A(M',N')$, we will set $f \otimesl_A g \in \text{Ext}^0_A(M \otimesl_A M', N \otimesl_A N')$ to be the image of $(f,g)$ under the aforementioned pairing. Moreover, if $N=N'=A$, we will compose with the multiplication isomorphism $A \otimesl_A A \to A$ so that the target is $\text{Ext}^0_A(M \otimesl_A M', A)$, and note that flatness conditions on our modules remove the necessity for derived tensor products.
\section{Indivisible sequences and descendability}
For $n \in \bb{N}$, let $[n] = \{1,2,...,n\}$ and $[\infty] = \bb{N}$.
\begin{definition}
Let $R$ be a commutative ring. For a set $S$, and a set-theoretic map
\[\Psi: S \to R\]
we associate the $R$-linear map
\[\Psi_{R,S}: \oplus_{S} R \to R \oplus \bigoplus_{S} R\]
defined by $\Psi_{R,S}(1_s)_{0}= 1$ and $\Psi_{R,S}(1_{s})_{s'} = \delta_{s'=s}$, where we regard the factor not indexed by $S$ on the right hand side as being in degree $0$.\\
\indent A \textit{$1$-indivisible sequence} is the data $\{S, \Psi: S \to R\}$ where $S$ is a set and $\Phi:S \to R$ a set-theoretic map such that the associated map $\Psi_{R,S}$ is an \textit{$R$-universally injective map}, in the sense that $\Psi_{R,S} \otimes_{R,f}R'$ is injective for any ring map $R \over{f}{\to} R'$. Equivalently, $\Psi_{R,S}$ is injective and $\coker{\Psi_{R,S}}$ is $R$-flat (as $\text{Tor}^1_R(\coker{\Psi_{R,S}}, R')=0$ for any ring map $R \to R'$).\\
\indent For any $n \in \bb{N} \cup \{\infty\}$, an \textit{$n$-indivisible sequence} is the data $\{\{S_i, \Psi_i:S_i \to R\}\}_{i \in [n]}$ where each $\{S_i, \Psi_i: S_i\to R\}$ is a $1$-indivisible sequence, and the elements $\{\Psi_i(s_i)\}_{i \in [n]}$ of $R$ do not generate the unit ideal for any choice of elements $s_i \in S_i$. An $\infty$-indivisible sequence of $R$ will simply be called an \textit{indivisible sequence} of $R$.
\end{definition}
While it may not be immediately obvious, there are many examples of rings $R$ with $n$-indivisible sequences. First, there is a general way to produce $n$-indivisible sequences from $1$-indivisible sequences.
\begin{proposition}\label{proposition:indivequiv}
Let $k$ be a commutative ring and $f: R \to R'$ a map of $k$-algebras, and for any $i \in [n]$, let $t_i: R \to R^{\otimes_k n}$ be the inclusion by $1$ into the $i^{\text{th}}$ factor.
\begin{enumerate}
\item[(i)] $\Psi_{R,S} \otimes_{R,f} R' = (f \circ \Psi)_{R',S}$. In particular, if $\{S, \Psi: S \to R\}$ is $1$-indivisible, then so is $\{S, f \circ \Psi: S \to R'\}$.
\item[(ii)] For each $n \in \bb{N} \cup \{\infty\}$ and $i \in [n]$, consider $1$-indivisible sequences $\{S_i, \Psi_i:S_i \to R\}$ such that for every $s_i \in S_i$, the composed map 
\[k \to R \to R/\Psi_i(s_i)\]
is faithfully flat. Then $\{\{S_i, t_i \circ \Psi_i: S_i \to R^{\otimes_k n}\}\}_{i \in [n]}$ is an $n$-indivisible sequence.
\end{enumerate}
\end{proposition}
\begin{proof}
(i) is standard. For (ii), by (i) the sequences $\{S_i, t_i \circ \Psi_i:S_i \to R^{\otimes_k n}\}$ are $1$-indivisible for each $i \in [n]$. Using the Kunneth spectral sequence applied to the complexes $R \over{\Psi(s_i)}{\to} R$, we can show that:
\[R^{\otimes_{k} n}/(t_1 \circ \Psi_1({s_1}),...,t_n \circ \Psi_n(s_n)) = \bigotimes_{i=1, k}^n R/\Psi_i(s_i).\]
for any elements $s_i \in S_i$. Under the faithfully flat assumption of the rings on the right, the latter ring is non-zero, hence (ii).
\end{proof}
\begin{example}\label{example:indiv}
Here we will list a few examples of $n$-indivisible sequences.
\begin{enumerate}
\item[(i)] Let $k$ be a commutative ring. Then $\{k, \Psi:k \to k[x]\}$, where $\Psi(a) = x-a$ for each $a \in k$, is a $1$-indivisible sequence. Using Proposition~\ref{proposition:indivequiv} we find that $\{\{k, \Psi_i: k \to k[x_1,...,x_n]\}\}_{i\in [n]}$, where $\Psi_i(a) = x_i - a$ for every $a \in k$, forms an $n$-indivisible sequence, for any $n \in \bb{N} \cup \{\infty\}$. 
\item[(ii)] Let $S$ be a set and $R$ be a ring with a subset $\{e_s\}_{s \in S}$ of orthogonal idempotents. Then the sequence $\{S, \Psi:S \to R\}$ defined by $\Psi(s) = 1-e_s$ is $1$-indivisible. Using Proposition~\ref{proposition:indivequiv}, we find that $\{\{S, t_i \circ \Psi:S \to R^{\otimes_{\bb{F}_p} n}\}\}_{i \in [n]}$, forms an $n$-indivisible sequence for any $n \in \bb{N}\cup \{\infty\}$. We remark that $R^{\otimes_{\bb{F}_p} n}$ is a $p$-boolean ring (see \ref{definition:pboolean}).
\end{enumerate}
Let us now briefly discuss why each of the examples are $1$-indivisible. For (i), we note that
\[\coker{\Psi_{k[x],k}} = \sum_{a \in k} \frac{1}{x-a}k[x] \subset \text{Frac}(k[x]).\]
It's standard to show that $\coker{\Psi_{k[x],k}}$ is therefore flat. Since $k[x]$ is a domain, $\Psi_{k[x],k}$ is clearly injective, so $\Psi_{k[x],k}$ is universally injective.\\
\indent For (ii), by (i) Proposition~\ref{proposition:indivequiv}, it suffices to show that $\Psi_{R,S}$ is injective. In particular, suppose we have an equation:
\[\Psi_{R,S}(\sum_{i \in [n]} r_i 1_{s_i}) = (\sum_{i \in [n]} r_i) \oplus \bigoplus_{i \in [n]} \Psi(s_i)r_i = 0\]
for a subset of distinct elements $\{s_1,...,s_n\} \subset S$ and $r_i \in R$. Therefore, $r_i \in e_{s_i}R, \; \; \forall i \in [n]$, and 
\[\sum_{i \in [n]} r_i = 0,\]
which, after multiplying by $e_{s_i}$, gives that $r_ie_{s_i}=0,$ and therefore $r_i \in (1-e_{s_i})R \Rightarrow r_i=0 \;\; \forall i \in [n]$, as desired. 
\end{example}
We have a key property of $n$-indivisible sequences.
\begin{theorem}\label{theorem:indivisiblecup}
Let $n \in \bb{N}$, $R$ be a ring and $\{\{S_i, \Psi_i: S_i \to R\}\}_{i \in [n]}$ an $n$-indivisible sequence in $R$. Then the class:
\[\bigotimes_{i \in [n],R} (R(\iota_{S_i}) \circ \cl{\Psi_{R,S_i}}) \in \text{Ext}^{n}_{R}(\otimes_{i \in [n],R} \coker{\Psi_{R,S_i}}, R\{\underset{i \in [n]}{\amalg} S_i\})\]
is non-zero as soon as $|S_i| \ge \aleph_{i-1}, \; \forall i \in [n]$.\footnote{Note that:
\[\bigotimes_{i\in [n],R}^n (R(\iota_{S_i}) \circ \cl{\Psi_{R,S_i}}) \neq 0 \Longleftrightarrow \bigotimes_{i\in [n],R} (\cl{\Psi_{R,S_i}}) \neq 0\]
since for each $i \in [n]$, $R({\iota}_{S_i}): \oplus_{S_i} R \to R\{S_i\}$ admits an $R$-linear left-inverse.}
\end{theorem}
In general, showing $\text{Ext}^n$ classes are non-zero is difficult. Recall the following lemma \cite[Lemma 2.15]{zelich}
\begin{lemma}\label{lemma:extcrit}
Let $A$ be a ring, $P_1, P_2, P_1', P_2', ..., P_r'$ be projective $A$-modules, together with two flat $A$-modules $M$ and $M'$ that form two exact sequences:
\[0 \to P_1 \stackrel{d_1}{\to} P_2 \to M \to 0,\]
\[0 \to P_1' \stackrel{d_2}{\to} ... \to P_r' \to M' \to 0.\]
We may view these as extensions $\eta_1 \in \text{Ext}^1_A(M, P_1)$ and $\eta_2 \in \text{Ext}^{r-1}_A(M',P_1')$, and as projective resolutions $P^{\bullet}_1, P^{\bullet}_2$. Then then the projective resolution:
\[P^{\bullet}_1 \otimes_A P^{\bullet}_2 \to M \otimes_A M',\]
gives rise to an extension $\text{Ext}^{r}_A(M \otimes_A M', P_1 \otimes_A P_1')$ which is equal to $\eta_1 \otimes_A \eta_2$. Moreover, showing this extension is non-zero is equivalent to verifying that 
\[P_1 \otimes P_1' \stackrel{d_1 \otimes 1_{P'_1} \oplus 1_{P_1} \otimes d_2}{\to} P_2 \otimes_A P_1' \oplus P_1 \otimes_A P_2\]
doesn't admit an $A$-linear left inverse.
\end{lemma}
\begin{proof}[Proof of Theorem~\ref{theorem:indivisiblecup}]
Using Lemma~\ref{lemma:extcrit}, we may reduce to showing there does not exist an $R$-linear extension $r$ fitting into the commutative diagram:
\[\begin{tikzcd}R\{\underset{i \in [n]}{\amalg} S_i\} &[5em]\\[5ex] \underset{{\underset{i \in [n]}{\prod} S_i}}{\bigoplus} R \arrow[r, "\bigoplus_i \Psi_{R,S_i}^n"] \arrow[u,"\underset{i \in [n],R}{\bigotimes} R(\iota_{S_i})"]& \bigoplus_{i} \underset{\underset{j \in [n] \setminus \{i\}}{\prod} S_j}{\bigoplus} R \oplus \underset{\underset{j \in [n]}{\prod}S_j}{\bigoplus}  R \arrow[ul, dashed, swap, "r"]\end{tikzcd}\]
Here, $\Psi^n_{R, S_i}$ is defined as the unique $R$-linear map satisfying:
\[\Psi^n_{R, S_i}(1_s) = 1_{p^n_i(s)} \oplus \Psi_i(s_i) 1_{s}, \; \; \forall s \in \prod_{j \in [n]} S_j.\]
Suppose such an extension $r$ existed. For every $ i \in [n]$, $s' \in \prod_{j \in [n]\setminus \{i\}} S_j$ and $s \in \prod_{j \in [n]}S_j$, let
\[r(1_{i,s'}) = f_i(s'), r(1_{i,s}) = g_i(s).\]
Here, $f_i(s'), g_i(s)$ are elements $R\{\amalg_{i \in [n]} S_i\}$, which satisfy the following equation in $R\{\amalg_{i \in [n]} S_i\}$: 
\[\sum_i f_i(p^n_i(t)) + \sum_i \Psi_{i}(\text{pr}_i(t)) g_i(t) = \prod_{i \in [n]} R(\iota_{S_i})(1_{\text{pr}_i(t)}), \; \; (*)\]
for every $t \in \prod_{j \in [n]} S_j$.\\
\indent Let 
\[d_n: R\{\amalg_{i \in [n]} S_i\} \to \underset{\underset{i \in [n]}{\prod}S_i}{\bigoplus}  R\]
be the $R$-linear left-inverse to ${\otimes}_{i \in [n],R} R(\iota_{S_i})$. Then for any $t' \in  \prod_{j \in [n]\setminus \{i\}} S_j$, $d_n(f_i(t'))$ is a vector in $\bigoplus_{\underset{i \in [n]}{\prod}S_i}  R$, which we view as an element of $\text{Hom}^{\text{fin}}_{\text{Set}}({\prod}_{i \in [n]}S_i , R)$. By Proposition~\ref{proposition:setbounds}, there exists an element $s \in \prod_{i \in [n]} S_i$ such that $d_n(f_i(p^n_i(s)))(s)=0, \; \forall i \in [n]$. Therefore, applying $d_n$ to $(*)$, we see that:
\[\sum_i \Psi_i(\text{pr}_i(s)) d_n(g_i(s))(s) = d_n(\prod_{i \in [n]} R(\iota_{S_i})(1_{\text{pr}_i(s)}))(s) = 1.\]
But this implies that $\{\Psi_i(\text{pr}_i(s))\}_{i \in [n]}$ generate the unit ideal in $R$, which is a contradiction.
\end{proof}
The following set-theoretic result was used above.
\begin{proposition}\label{proposition:setbounds}
Let $R$ be a ring and $S$ a set. For each $ i\in \{1,...,n\}$, consider functions:
\[v_i: \prod_{j \in [n] \setminus \{i\}} S_j \to \text{Hom}^{\text{fin}}_{\text{Set}}(\prod_{i \in [n]} S_i,R).\]
If $|S_i| \ge \aleph_{i-1}$, then there exists an element $s \in \prod_{i \in [n]} S_i$ such that $v_i(p^n_i(s))(s) = 0$ for each $i \in \{1,2,...,n\}$.
\end{proposition}
\begin{proof}
Define $v'_i = v_i(p^n_i(s))(s): \prod_{i \in [n]} S_i \to R$. The functions $v'_i$ have the property that for any $s' \in \prod_{j \in  [n] \setminus \{i\}} S_i$, 
\[v'_i \circ t_i^{s'}=v_i(s') \circ t_i^{s'} \in \text{Hom}^{\text{fin}}_{\text{Set}}(S,R).\]
For each $s' \in \prod_{i \in [n-1]} S_i$, define a subset $V_{s'} \subset S_n$ as follows:
\[V_{s'} = \{s \in S_n | \; v'_n(t_n^{s'}(s)) \neq 0 \}.\]
By supposition, $V_{s'}$ is a finite subset of $S_n$. Notice that:
\[|\cup_{s' \in \prod_{i \in [n-1]} S_i} V_{s'}| \le \aleph_0 \aleph_{n-2} < \aleph_{n-1}.\]
Therefore, there is an $s_{n} \in S_n$ such that:
\[v'_n(t_n^{s'}(s_n)) = 0, \; \; \forall s' \in \prod_{i \in [n-1]} S_i.\]
We may repeat this argument for the restricted functions $v'_i \restriction \prod_{j \in [n] \setminus \{i\}} S_j \times s_n$, and we eventually end up with elements $s_i \in S_i$ defining an element $s \in \prod_{i \in [n]}S_i$ such that $v'_i(s) = v_i(p^n_i(s))(s)=0, \; \; \forall i \in [n]$, as desired.\\
\end{proof}
The following translates Theorem~\ref{theorem:indivisiblecup} to an algebraic result.
\begin{theorem}\label{theorem:indivalg}
Let $n \in \bb{N} \cup \{\infty\}$, and $\{\{S_i, \Psi_i:S_i \to R\}\}$ be a $n$-indivisible sequence in $R$. For all $i \in [n]$, suppose we have $R$-algebra maps $R\{S_i\} \over{g_i}{\leftarrow} A_i \over{f_i}{\to} B_i$ fitting into a commutative diagram:
\[\begin{tikzcd}
R\{S_i\}&\\
A_i \arrow[u, "g_i", swap] \arrow[r, "f_i"] & B_i \\
\oplus_{S_i} R \arrow[u] \arrow[uu, bend left, "R(\iota_{S_i})"] \arrow[r, "\Psi_{R,S_i}"] & R \oplus \bigoplus_{S_i} R \arrow[u]
\end{tikzcd}\]
and each $f_i$ are faithfully flat. If for all $i \in [n]$, $|S_i| \ge \aleph_{i-1}$, then the faithfully flat ring map:
\[\begin{tikzcd}\bigotimes_{i \in [n], R} A_i \arrow[r, "\otimes_{i \in [n],R} f_i"]& \bigotimes_{i \in [n],R} B_i\end{tikzcd}\]
has descendability exponent $n$.
\end{theorem}
\begin{proof}
Let $h_i: R\{S_i\} \to T_i$ be the base-change of $f_i$ along $g_i$. By base-change along $\otimes_{i \in [n], R} g_i : \otimes_{i \in [n], R} A_i \to \otimes_{i \in [n], R} R\{S_i\}$, it suffices to that $\otimes_{i \in [n], R} h_i$ has exponent of descendability $n$. But this follows directly by combining Proposition~\ref{proposition:cup} (ii) and Theorem~\ref{theorem:indivisiblecup}.
\end{proof}
The following proposition was used above.
\begin{proposition}\label{proposition:cup}
Let $A \to B$ be a flat map rings.
\begin{enumerate}
\item[(i)] For any natural number $n$, let $M_i$ be $B$-modules in $D(B)$ for $i \in [n]$, and let $\mu_{B/A}^{k}: B^{\otimes_A k+1} \to B$ be the multiplication map realising $B$ as an $A$-algebra. Consider elements $\eta_i \in \text{Ext}^1_B(M_i, B)$, then:
\[\mu^{n-1}_{B/A} \circ \otimesl_{i,A} R_{B|A}(\eta_i) \neq 0 \implies \otimesl_{i, B} \eta_i \neq 0.\]
\item[(ii)] For every $s \in \bb{N}$, suppose we have flat $A$-algebras $B_s$, faithfully flat maps $f_s: B_s \to C_s$ of $A$-algebras and injective maps $\Psi_s: M_s \to N_s$ of flat $A$-modules fitting into a commutative diagram:
\[\begin{tikzcd}
B_s \arrow[r, "f_s"] & C_s \\
M_s \arrow[u, "h_s"] \arrow[r, "\Psi_s"] & N_s \arrow[u]
\end{tikzcd}\]
Assume the unit $A \to B_s$ admits an $A$-linear left-inverse $k_s: B_s \to A$. Then for any subset $S \subset \bb{N}$, and any finite subset $S' \subset S$:
\[\bigotimes_{s' \in S',A} (h_{s'} \circ \cl{\Psi_{s'} }) \neq 0 \implies \underset{s' \in S', \otimes_{s \in S,A}B_s}{\bigotimes} \cl{ \otimes_{s \in S,A} f_s} \neq 0.\]
\end{enumerate}
\end{proposition}
\begin{proof}
For (i), we remark more generally that for any pair of morphisms $f: N_1 \to K_1, g: N_2 \to K_2$ in $D(B)$, the bar construction of \cite[Chapter 4.4.2]{LurieHA} gives a homotopy coherent diagram:
\[\begin{tikzcd}
N_1 \otimesl_B K_1 \arrow[r, "f \otimesl_B g"] & N_2 \otimesl_B K_2 \\
N_1 \otimesl_A K_1 \arrow[u] \arrow[r, "f \otimesl_A g"] & N_2 \otimesl_A K_2 \arrow[u]
\end{tikzcd}\]
Applying this to the $\eta_i$ repeatedly, we end up with a homotopy commutative diagram:
\[\begin{tikzcd}
\otimesl_{i, B} M_i[-1] \arrow[r, "\otimesl_{i,B} \eta_i"] & B^{\otimesl_B n} \arrow[r, "\simeq"] & B \\
\otimesl_{i, A} M_i[-1] \arrow[u, "\varphi"] \arrow[r,swap, "\otimesl_{i,B} R_{B|A}(\eta_i)", {yshift=-3pt}] & B^{\otimesl_A n} \arrow[u] \arrow[ru, swap,"\mu^{n-1}_{B/A}"] &
\end{tikzcd}\]
We therefore see that:
\[\otimesl_{i,B} \eta_i =0 \implies (\otimesl_{i,B} \eta_i) \circ \varphi =0 \implies \mu^{n-1}_{B/A} \circ \otimesl_{i,A} R_{B|A}(\eta_i) = 0.\]
For (ii), let $B_{S} = \otimes_{s \in S} B_{s}$,  $i_s: B_s \to B_S$ be the inclusion in the $s^{\text{th}}$-factor, and $g_s: \coker{f_s} \to \coker{\otimes_{s \in S} f_s}$ the induced map. Note that:
\[\bigotimes_{s' \in S',A} (h_{s'} \circ \cl{\Psi_{s'} }) \neq 0 \implies \bigotimes_{s' \in S',A} \cl{f_{s'}} \neq 0.\]
By (i) it suffices to show that:
\[\mu^{|S'|-1}_{B_{S}/A} \circ \bigotimes_{s' \in S', A} R_{B_{S}|A}(\cl{\otimes_{s \in S} f_s}) \neq 0.\]
This follows from the following calculation:
\begin{align*}k_{S'} \circ \mu^{|S'|-1}_{B_{S}/A} \circ \bigotimes_{s' \in S', A} R_{B_{S}|A}(\cl{\otimes_{s \in S} f_s}) \circ \bigotimes_{s' \in S', A} g_{s'} &= k_{S'} \circ \mu^{|S'|-1}_{B_{S}/A} \circ \bigotimes_{s' \in S', A} (R_{B_{S}|A}(\cl{\otimes_{s \in S} f_s}) \circ g_{s'})\\
& = k_{S'} \circ \mu^{|S'|-1}_{B_{S}/A} \circ \bigotimes_{s' \in S'} (i_{s'} \circ \cl{f_{s'}})\\
&= \left (k_{S'} \circ \mu^{|S'|-1}_{B_{S}/A} \circ \bigotimes_{s' \in S'} i_{s'}\right) \circ \bigotimes_{s' \in S'} \cl{f_{s'}}\\
&= \bigotimes_{s' \in S'} \cl{f_{s'}} \\
&\neq 0,
\end{align*}
where $k_{S'}: \otimes_{s \in S} R \to \otimes_{s \in S \setminus S'}R$ is defined as the tensor product of the $A$-linear maps $\text{id}_{B_{s''}}: B_{s''} \to B_{s''}$ and $k_{s'}$ for each $s'' \in S \setminus S'$ and $s' \in S'$. \\
\end{proof}
Before stating our main result, we need to recall a few definitions.
\begin{definition}\label{definition:pboolean}
A ring $R$ is said to be \textit{p-boolean} if $R$ is in characteristic $p$ and for every $r \in R, r^p=r$.
\end{definition}
The following is a discrete version of a derived result appearing in \cite[Theorem 4.6]{pbool}.
\begin{theorem}\label{theorem:p-booleanmonadic}
Let $R$ be a $p$-boolean ring. Let $\text{CAlg}_{R}$ (resp. $\text{CAlg}^{\text{perf}}_{R}$, $\text{CAlg}^{\phi=1}_{R}$) denote the category of (commutative) $R$-algebras (resp. perfect $R$-algebras, $p$-boolean $R$-algebras). Each of the forgetful functors:
\[\text{CAlg}^{\phi=1}_{R} \to \text{CAlg}^{\text{perf}}_{R} \to \text{CAlg}_{R} \to \text{Mod}_R\]
are monadic. In particular, the left adjoint of this forgetful functor is the functor $F_R^p: \text{Mod}_R \to \text{CAlg}^{\phi=1}_{R}$ defined on objects by the rule:
\[M \mapsto \text{Sym}_R(M) \mapsto \text{Sym}_R(M)^{\text{perf}} \mapsto \text{Coeq}(\begin{tikzcd}\text{Sym}_R(M)^{\text{perf}} \ar[r,shift left=.5ex,"\phi"]
  \ar[r,shift right=.5ex,swap,"\text{id}"] & \text{Sym}_R(M)^{\text{perf}} \end{tikzcd}),\]
where $(-)^{\text{perf}}$ is the colimit perfection.
\end{theorem}
\begin{proposition}\label{proposition:symformula}
Let $R$ be a ring and $f: M \to N$ an injective map of flat $R$-modules such that $K=\coker{f}$ is $R$-flat. Then 
\[\text{Sym}_R(f): \text{Sym}_R(M) \to \text{Sym}_R(N)\]
is a faithfully flat map of $R$-algebras. Furthermore, if $R$ is $p$-boolean, then $F_R^p(f)$ is a faithfully flat map of $p$-boolean rings.
\end{proposition}
\begin{proof}
By Lazard's theorem, $K=\colim{i\in I} K_i$ where $I$ is a filtered indexing category and $K_i$ is a finite free $R$-module. If we define $N_i = N \times_{K} K_i$, then
\[\colim{i \in I} N_i = \colim{i \in I} N \times_{K} K_i = N \times_{K} \colim{i \in I} K_i = N,\]
and we have an exact sequence
\[0 \to M \over{f_i}{\to} N_i \to K_i \to 0,\]
which is split. Therefore:
\[\text{Sym}_R(N_i) = \text{Sym}_R(M) \otimes_R \text{Sym}_R(K_i)\]
and if $R$ is $p$-boolean:
\[F^p_R(N_i) = F_R^p(M) \otimes_R F_R^p(K_i).\]
Note that as $K_i$ are finite free, the unit $R \to \text{Sym}_R(K_i)$ is faithfully flat, and if $R$ is futher assumed to be $p$-boolean, then likewise for $R \to F^p_R(K_i)$.\footnote{If $K_i =R^{\oplus n}$, then $F^p_R(K_i) = R[X_1,...,X_n]/(X_1^p-X_1, ..., X_n^p-X_n)$.} Hence the map:
\[\text{Sym}_R(M) \over{\text{Sym}_R(f_i)}{\to} \text{Sym}_R(N_i)\]
is faithfully flat, and if $R$ is $p$-boolean, then:
\[F^p_R(M) \over{F^p_R(f_i)}{\to} F^p_R(N_i)\]
is also faithfully flat. The result follows by noting $\text{Sym}_R(f) = \colim{i\in I}\text{Sym}_R(f_i)$ and if $R$ is $p$-boolean, $F^p_R(f) = \colim{i \in I} F^p_R(f_i)$.
\end{proof}
Our main result on descendability of faithfully flat ring maps can then be stated as follows.
\begin{theorem}\label{theorem:mainindiv}
Let $n \in \bb{N} \cup \{\infty\}$, $R$ be a commutative ring and $\{\{S_i, \Psi_i: S_i \to R\}\}_{i \in [n]}$ an $n$-indivisible sequence, and $|S_i| \ge \aleph_{i-1}$. Then the ring map 
\[\bigotimes_{i=1, R}^n \text{Sym}_R(\Psi_{R,S_i}): \bigotimes_{i=1, R}^n \text{Sym}_R(\oplus_{S_i} R) \to \bigotimes_{i=1, R}^n \text{Sym}_R(R \oplus \bigoplus_{S_i} R),\]
is faithfully flat with exponent of descendability $n$. In fact, the ring map:
\[\begin{tikzcd}R\{\underset{i=1}{\amalg^n}S_i\}\arrow[r, "\underset{i=1,R}{\bigotimes^n} \text{Sym}_R(\Psi_{R,S_i}) \underset{\underset{i=1,R}{\bigotimes^n} \text{Sym}_R(\oplus_{S_i} R)}{\bigotimes}  R\{\underset{i=1}{\amalg^n}S_i\}"] &[8em] \underset{i=1,R}{\bigotimes^n} \text{Sym}_R(R \oplus \bigoplus_{S_i} R)  \underset{\underset{i=1,R}{\bigotimes^n} \text{Sym}_R(\oplus_{S_i} R)}{\bigotimes}R\{\underset{i=1}{\amalg^n}S_i\}\end{tikzcd},\]
also has descendability exponent $n$ too.\\
\indent If $R$ is $p$-boolean, then the ring map:
\[\bigotimes_{i=1, R}^n F^p_R(\Psi_{R,S_i}): \bigotimes_{i=1, R}^n F^p_R(\oplus_{S_i} R) \to \bigotimes_{i=1, R}^n F^p_R(R \oplus \bigoplus_{S_i} R),\]
is a faithfully flat of $p$-boolean rings of descendablity exponent $n$.
\end{theorem}
\begin{proof}
After noting commutative diagrams:
\[\begin{tikzcd} \oplus_{S_i} R \arrow[r] \arrow[rd]\arrow[rr, bend left, "R(\iota_{S_i})"]&  \text{Sym}_R(\oplus_{S_i} R) \arrow[r]& R\{S_i\}\\
& F^p_R(\oplus_{S_i} R) \arrow[ru]&
\end{tikzcd}\]
the result follows from combining Proposition~\ref{proposition:symformula} and Theorem~\ref{theorem:indivalg}.
\end{proof}
\begin{corollary}\label{corollary:mainindiv}
For any $n \in \bb{N} \cup \{\infty\}$:
\begin{enumerate}
\item[(i)] There exists a faithfully flat ring map $A \to A'$ between $\aleph_{n-1}$-Noetherian rings, with $A$ of Krull-dimension $n$, which has descendability exponent $n$. 
\item[(ii)] There exists a faithfully flat ring map $A \to A'$ between $\text{min}(\beth_{n-1}, 2^{\aleph_{n-1}})$-countable $p$-boolean rings which has descendability exponent $n$.
\end{enumerate}
\end{corollary}
\begin{proof}
For (i), we may use the first example of \ref{example:indiv}. In particular, if we set $R = k[x_1,x_2,...,x_n]$, then from Theorem~\ref{theorem:mainindiv} we may conclude that there is a faithfully flat map:
\[k\{k^{\amalg n}\}[x_1,...,x_n] = R\{k^{\amalg n}\} \to R\{k_{+}^{\amalg n}\}= A'\]
which has descendability exponent $n$ when $|k| \ge \aleph_{n-1}$, where the set $k_+ = \{*\}\amalg k$. Furthermore, it's clear that $R\{k^{\amalg n}\}$ has Krull-dimension $n$, and that the rings are $\aleph_{n-1}$-Noetherian.\\
\indent For (ii), by using Theorem~\ref{theorem:mainindiv} and Proposition~\ref{proposition:indivequiv}, it suffices to find a $1$-indivisible sequence $\Psi: S \to R$ where the image of $S$ is a set of orthogonal idempotents in $R$ with $|S|=\aleph_{n-1}$ and $R$ is a $p$-boolean ring with $|R|=2^{\aleph_{n-1}}$ or $\beth_{n-1}$.\\
\indent If $S$ is an arbitrary set with $|S|=\aleph_{n-1}$, then $R= \text{Hom}_{\text{Set}}(S, \bb{F}_p)$ is a $p$-boolean ring with $|R|=2^{\aleph_{n-1}}$ containing the set $\{\delta_{t=s}\}_{s \in S}\subset R$ of orthogonal idempotents.
\begin{claim}
There exists a $p$-boolean ring $R$ with $|R| = \beth_n$ and a subset $S \subset R$ consisting of orthogonal idempotents with $|S|=\beth_n$.
\end{claim}
\begin{proof}
We will prove the claim by induction on $n$. In the case $n=0$, we remark that the $p$-boolean ring $R_0=F_{\bb{F}_p}^p(R_0^{\oplus \bb{N}})=\bb{F}_p[X_1,X_2,...]/(X_1^p-X_1,X_2^p-X_2,...)$ admits a countably infinite subset of orthogonal idempotents given by a sequence $S_0=\{a_n\}_{n \in \bb{N}}$ of elements of $R$ defined by $a_n = (1-X^{p-1}_n) \prod_{0 \le i \le n-1} X^{p-1}_{i}$. The sequence $\{S_0, a: S_0 \to R_0\}$ solves the $n=0$ base-case.\\
\indent Suppose by induction we have a ring $R_{n-1}$ such that $|R_{n-1}| = \beth_{n-1}$ and that there is a subset $S_{n-1} \subset R_{n-1}$ consisting of orthogonal idempotents with $|S_{n-1}| = \beth_{n-1}$. Consider the ring:
\[R_n=\prod_{S_{n-1}}R_{n-1}.\]
We note $|R_n| = \beth_{n-1}^{\beth_{n-1}} = \beth_n$, and the set
\[S_n = \prod_{S_{n-1}} S_{n-1} \subset R_n\]
consists of orthogonal idempotents and $|S_n|=\beth_{n-1}^{\beth_{n-1}}=\beth_n$, as desired.
\end{proof}
\end{proof}

\bibliographystyle{amsalpha}
\bibliography{ffdescent}
\end{document}